\documentclass[]{article}

\usepackage[T1]{fontenc}
\usepackage{amsmath,amssymb,amsthm,mathrsfs}
\usepackage{pxfonts}
\usepackage{sectsty}
\usepackage{microtype}
\usepackage{fancyhdr}
\usepackage[hidelinks]{hyperref}
\hypersetup{
  pdftitle={On non-monotonicity of logarithmic energy for random matrices},
  pdfauthor={Theodoros Assiotis}
}
\allowdisplaybreaks
\sectionfont{\scshape\centering\fontsize{11}{14}\selectfont}
\subsectionfont{\scshape\fontsize{11}{14}\selectfont}
\newcommand\shorttitle{Non-monotonicity of log-energy}
\newcommand\authors{T. Assiotis}
\newtheorem{thm}{Theorem}[section]
\newtheorem{lem}[thm]{Lemma}
\newcommand{\E}{\mathbb{E}}
\newcommand{\Prob}{\mathbb{P}}
\newcommand{\R}{\mathbb{R}}
\newcommand{\C}{\mathbb{C}}
\newcommand{\cE}{\mathfrak{E}}
\newcommand{\cF}{\mathfrak{F}}
\newcommand{\1}{\mathbf{1}}
\newcommand{\Tr}{\operatorname{Tr}}
\newcommand{\Col}{\operatorname{Col}}
\newcommand{\pentry}{\mathfrak{p}}
\newcommand{\qentry}{\mathfrak{q}}

\title{\large\bf ON NON-MONOTONICITY OF LOGARITHMIC ENERGY\\
FOR RANDOM MATRICES}
\author{\small THEODOROS ASSIOTIS}
\date{}

\begin{document}
\maketitle

\begin{abstract}
We construct a finite-energy Wigner counterexample and a concrete one-parameter family of Gaussian-regularised Bernoulli entry laws yielding counterexamples to the conjectural dimensional monotonicity of the quadratically penalised logarithmic energy for mean empirical spectral distributions by Chafa\"i, Dadoun, and Youssef \cite{ChafaiDadounYoussef}.
\end{abstract}

\section{Introduction}\label{SectionIntroduction}

\subsection{Background}

Monotonicity of entropy-like functionals is a recurrent mechanism by
which an evolution or an approximation is driven towards equilibrium.
The Boltzmann H-theorem and the free-energy dissipation principle
for ergodic Markov dynamics are classical instances of this mechanism;
see for example \cite{Chafai2015,Villani2008}.  In central-limit problems, entropy
is monotone under normalised convolution in the classical setting
\cite{ArtsteinBallBartheNaor2004,MadimanBarron2007} and admits free
analogues \cite{DadounYoussef2021,Shlyakhtenko2007}.  These examples
combine a conserved moment constraint with a variational
characterisation of the limiting law.

Random-matrix limit laws have a parallel variational structure.  The
semicircle, circular, and Marchenko--Pastur laws arise as equilibrium
measures for logarithmic energies with external fields; see
\cite{AndersonGuionnetZeitouni2010,Forrester2010,HiaiPetz2000,SaffTotik1997}.
For Gaussian Wigner matrices and suitable \(\beta\)-ensembles, the same
penalised logarithmic energies occur as large-deviation rate
functionals
\cite{AndersonGuionnetZeitouni2010,BenArousGuionnet1997,Forrester2010}.
On the real line,
logarithmic energy is also closely linked to Voiculescu's free entropy
\cite{HiaiPetz2000,Voiculescu2005}, which makes its finite-dimensional
behaviour along mean empirical spectral distributions a natural
object of study.

Motivated by these analogies, Chafa\"i, Dadoun, and Youssef formulated very intriguing
dimensional monotonicity questions for the mean empirical spectral
distributions in the Wigner, circular-law, and
Marchenko--Pastur settings  in \cite{ChafaiDadounYoussef}.  They proved the
corresponding monotonicity for the Gaussian unitary, complex Ginibre,
and square Laguerre unitary ensembles by beautiful explicit computation.  Moreover, they provided extensive numerical evidence in support of this monotonicity for various models. 

The purpose of this paper is
to show, by a finite-energy Wigner counterexample and a concrete
one-parameter Gaussian-regularised Bernoulli family, that the proposed
monotonicity in its strongest form does not hold for general Wigner matrices or general
matrices with i.i.d.\ entries. Nevertheless, it is a very interesting question to understand whether such a monotonicity property holds under more restrictive assumptions or whether it holds after a large enough dimension threshold.

We now move on to make our results precise.

For a probability measure $\mu$ on $X=\R$ or $\C$ with finite second
moment, write
\[
 m_2(\mu)\overset{\mathrm{def}}{=}\int_X |x|^2\,\mu(\mathrm{d}x),
 \qquad
 \cE(\mu)\overset{\mathrm{def}}{=}
 \iint_{X\times X}\log\frac1{|x-y|}\,\mu(\mathrm{d}x)\mu(\mathrm{d}y).
\]
The negative part of the logarithmic kernel is integrable under this
moment assumption, so $\cE(\mu)\in(-\infty,+\infty]$ is well defined.
The relevant penalised energies are
\[
 \cF_{\R}(\mu)\overset{\mathrm{def}}{=}\frac12m_2(\mu)+\cE(\mu),
 \qquad
 \cF_{\C}(\mu)\overset{\mathrm{def}}{=}m_2(\mu)+\cE(\mu).
\]
Their equilibrium measures are, respectively, the standard semicircle
law and the circular law,
\[
 \sigma(\mathrm{d}x)=\frac1{2\pi}\sqrt{(4-x^2)_+}\,\mathrm{d}x,
 \qquad
 \omega(\mathrm{d}z)=\frac1\pi\1_{\{|z|\leq1\}}\,\mathrm{d}^2z.
\]
They satisfy
\begin{equation}\label{EqEquilibriumValues}
 m_2(\sigma)=1,\quad m_2(\omega)=\frac12,\quad
 \cE(\sigma)=\cE(\omega)=\frac14,\quad
 \cF_{\R}(\sigma)=\cF_{\C}(\omega)=\frac34.
\end{equation}
For an $n\times n$ matrix $\mathbf{A}$, let
\[
 \mu_{\mathbf{A}}\overset{\mathrm{def}}{=}
 \frac1n\sum_{j=1}^n\delta_{\lambda_j(\mathbf{A})}
\]
be its empirical spectral distribution, with algebraic multiplicity.

By a Wigner sequence we mean
$\mathbf{M}_n=(\mathsf{W}_{ij})_{1\leq i,j\leq n}$ obtained from an
infinite Hermitian array whose upper-triangular entries are independent,
whose strictly upper-triangular entries have a common centered law with
$\E|\mathsf{W}_{12}|^2=1$, and whose diagonal entries have a common real
law with finite second moment. In the i.i.d.\ setting we restrict to
infinite arrays of independent copies of a complex random variable
$\mathsf{X}$ satisfying $\E\mathsf{X}=0$ and $\E|\mathsf{X}|^2=1$.
Question~1.1 of \cite{ChafaiDadounYoussef}, and the centered case of
their Question~1.2, ask whether
\begin{align}
 \cF_{\R}\left(\E[\mu_{\mathbf{M}_n/\sqrt n}]\right)
 &\searrow\cF_{\R}(\sigma)
 &&\text{for Wigner matrices},\label{EqQuestionReal}\\
 \cF_{\C}\left(\E[\mu_{\mathbf{M}_n/\sqrt n}]\right)
 &\searrow\cF_{\C}(\omega)
 &&\text{for matrices with i.i.d.\ entries}.\label{EqQuestionComplex}
\end{align}
In the Wigner question one additionally imposes
$m_2(\E[\mu_{\mathbf{M}_n/\sqrt n}])=1$.

\subsection{Main results}\label{SectionMainResults}

The first result disproves Wigner dimensional monotonicity by starting
at the semicircle equilibrium law and producing a strictly larger
energy in dimension two, with all energies finite.
\begin{thm}\label{ThmWignerMain}
Let $(\mathsf{W}_{ij})_{1\leq i\leq j<\infty}$ be independent with
common law $\sigma$, set $\mathsf{W}_{ji}=\mathsf{W}_{ij}$ for $i<j$,
and define
\[
 \mathbf{M}_n=(\mathsf{W}_{ij})_{1\leq i,j\leq n},
 \qquad
 \rho_n=\E[\mu_{\mathbf{M}_n/\sqrt n}].
\]
Then $(\mathbf{M}_n)_{n\geq1}$ is a real symmetric Wigner sequence,
each $\rho_n$ has a compactly supported bounded density, and
\[
 m_2(\rho_n)=1,\qquad \cE(\rho_n)\in\R,\qquad n\geq1.
\]
Moreover, $\rho_1=\sigma$, $\rho_2((2,\infty))>0$, and, with
$\pentry_0=\sigma((3/2,2))$,
\[
 \cF_{\R}(\rho_2)-\cF_{\R}(\rho_1)
 \geq\frac{\pentry_0^3}{2}\left(\frac34-\log2\right)>0,
 \qquad
 \cF_{\R}(\rho_1)=\frac34.
\]
\end{thm}

The second result disproves dimensional monotonicity for i.i.d.\
matrices with smooth entry laws and finite-energy mean spectral
distributions. Its proof compares eigenvalue clustering in dimensions
two and three as the Gaussian regularisation vanishes.
\begin{thm}\label{ThmComplexMain}
Let $\pentry=1/256$ and $\qentry=255/256$. Let $\mathsf{X}$ and
$\mathsf{G}$ be independent, with
\begin{equation}\label{EqBernoulliLaw}
 \Prob\left(\mathsf{X}=-\frac1{\sqrt{255}}\right)=\qentry,
 \qquad
 \Prob(\mathsf{X}=\sqrt{255})=\pentry,
 \qquad
 \Prob(\mathsf{G}\in\mathrm{d}z)=\frac1\pi \mathrm{e}^{-|z|^2}\,\mathrm{d}^2z.
\end{equation}
For $\varepsilon>0$, put
\[
 \mathsf{Y}_\varepsilon=
 \frac{\mathsf{X}+\varepsilon\mathsf{G}}{\sqrt{1+\varepsilon^2}},
\]
let $\mathbf{M}_{n,\varepsilon}$ have i.i.d.\ entries with this law,
and set
\[
 \nu_{n,\varepsilon}=
 \E[\mu_{\mathbf{M}_{n,\varepsilon}/\sqrt n}].
\]
Then $\mathsf{Y}_\varepsilon$ is centered, satisfies
$\E|\mathsf{Y}_\varepsilon|^2=1$, and has a strictly positive
real-analytic density, all moments, and sub-Gaussian tails. Moreover,
\[
 \cE(\nu_{n,\varepsilon})\in\R,\qquad n\geq1,\quad\varepsilon>0,
\]
and there exists $\varepsilon_0>0$ such that
\[
 \cF_{\C}(\nu_{3,\varepsilon})>\cF_{\C}(\nu_{2,\varepsilon}),
 \qquad 0<\varepsilon<\varepsilon_0.
\]
\end{thm}

The Wigner counterexample exploits the fact that the semicircle law
already minimizes the quadratically penalised logarithmic energy.
By choosing semicircular entries, we arrange that the mean spectral
distribution in dimension one is exactly this equilibrium measure.
In dimension two, however, an event on which all three independent
entries are sufficiently large and positive forces an eigenvalue
beyond the semicircle support. The effective potential is strictly
positive outside this support, so this event yields a quantitative
increase in energy. Since the second moment remains equal to one,
the increase also occurs in the logarithmic energy itself.

The i.i.d.\ counterexample instead exploits a different structure, that of eigenvalue clustering.
The Bernoulli entry law takes a common value $c$ with probability
close to one, while a rare value ensures centering and unit second
moment. On the event that all entries take the common value, the
normalized matrix is $c\mathbf{J}_n/\sqrt n$, where $\mathbf{J}_n$
is the all-ones matrix. This normalized matrix has eigenvalue zero
of multiplicity $n-1$ and one eigenvalue $c\sqrt n$. The corresponding spectral masses
are therefore $1/2,1/2$ in dimension two and $2/3,1/3$ in dimension
three.

Small Gaussian regularisation turns these atoms into narrow eigenvalue
clusters and makes the logarithmic energies finite. Heuristically,
a cluster of mass $a$ at scale $\varepsilon$ contributes at leading
order $a^2\log(1/\varepsilon)$: the square appears because logarithmic
energy involves two independent samples from the mean spectral
distribution. The sums of squared masses for the all-common
configurations are $1/2$ and $5/9$, respectively. This difference
supplies the mechanism for the increase from dimension two to
dimension three, while the uniformly bounded second moments make the
quadratic penalty negligible on the logarithmic scale. 

Then, in dimension two, a
density bound from Gaussian smoothing controls the energy from above
through the squared atom masses of the limiting measure; convexity
then handles all exceptional Bernoulli configurations together.
In dimension three, a lower bound comes solely from the all-common
configuration, whose normality ensures that its two spectral clusters
remain at scale $\varepsilon$ under a bounded Gaussian perturbation.
The choice $\pentry=1/256$ makes the exceptional configurations
sufficiently unlikely that these elementary bounds already separate
the two dimensions, without any enumeration of their spectra (any choice $\pentry<1/216$ would make the argument go through as it is; the original counterexample used $\pentry=1/32$ which required significantly more technicalities related to enumeration).

\paragraph{Tool and computational resource disclosure} In the course of the research presented here I have been using AI tools, mainly ChatGPT 5.5, 5.6 Sol for ideation, technical help, editing and checking the proofs and general editing and proofreading of the manuscript, at the level of a co-author. Any mistakes are my own responsibility.

\section{Elementary logarithmic-energy bounds}\label{SectionPreliminaries}

The following bounds control the positive and negative parts of the
logarithmic kernel; they will ensure that the energies in both
constructions are finite. For $X=\R$ or $\C$ and a probability
measure $\mu$ on $X$ with finite second moment, write
\[
 I_+(\mu)=\iint_{X\times X}\log^+\frac1{|x-y|}\,\mu(\mathrm{d}x)\mu(\mathrm{d}y),
 \qquad
 I_-(\mu)=\iint_{X\times X}\log^+|x-y|\,\mu(\mathrm{d}x)\mu(\mathrm{d}y).
\]
The inequality $\log^+r\leq r^2/2$ gives
\begin{equation}\label{EqNegativePart}
 I_-(\mu)\leq\frac12\iint_{X\times X}|x-y|^2\,\mu(\mathrm{d}x)\mu(\mathrm{d}y)
 =m_2(\mu)-\left|\int_X x\,\mu(\mathrm{d}x)\right|^2
 \leq m_2(\mu).
\end{equation}
If $\mu$ also has bounded density $f$, then translation invariance and
Tonelli's theorem give
\[
 I_+(\mu)\leq
 \|f\|_\infty\int_{\{u\in X:|u|<1\}}\log\frac1{|u|}\,\mathrm{d}u<\infty.
\]
Here and below Lebesgue measure is understood in the ambient real
dimension. Thus a bounded density and a finite second moment imply
absolute integrability of the logarithmic kernel.

For a finitely supported probability measure $\alpha$ on $\C$, define
\[
 \Col(\alpha)\overset{\mathrm{def}}{=}
 \sum_{a\in\C}\alpha(\{a\})^2.
\]
This is a convex function of $\alpha$, since it is the squared
$\ell^2$ norm of its atom masses, and $\Col(\alpha)\leq1$.
Write $\mathscr{D}(a,r)=\{z\in\C:|z-a|<r\}$.

The next lemma bounds the leading logarithmic energy of a regularised
atomic measure using only its density and weak limit. It will provide
the upper bound in dimension two.
\begin{lem}\label{LemDensityUpper}
Let $(\mu_\varepsilon)_{0<\varepsilon<1}$ be probability measures on
$\C$ with uniformly bounded second moments and densities satisfying
$\|f_\varepsilon\|_\infty\leq C\varepsilon^{-2}$ for a fixed $C<\infty$.
Suppose that $\mu_\varepsilon$ converges weakly to a finitely supported
probability measure $\alpha$. Then
\begin{equation}\label{EqDensityUpper}
 \limsup_{\varepsilon\downarrow0}
 \frac{\cE(\mu_\varepsilon)}{\log(1/\varepsilon)}
 \leq\Col(\alpha).
\end{equation}
\end{lem}

\begin{proof}
Take independent $\mathsf{Z}_\varepsilon,\mathsf{W}_\varepsilon$ with
law $\mu_\varepsilon$, and put $L_\varepsilon=\log(1/\varepsilon)$.
Choose $0<\delta<1$ smaller than the distance between any two distinct
atoms of $\alpha$; if there is only one atom, choose any such $\delta$.
For $0<\varepsilon<\delta$ and $r>0$,
\[
 \log^+\frac1r
 \leq L_\varepsilon\1_{\{r\leq\delta\}}
      +\log\frac1\delta+\log^+\frac{\varepsilon}{r}.
\]
The density bound gives
\[
 \E\left[\log^+\frac{\varepsilon}
 {|\mathsf{Z}_\varepsilon-\mathsf{W}_\varepsilon|}\right]
 \leq C\varepsilon^{-2}
 \int_{\{u\in\C:|u|<\varepsilon\}}\log\frac{\varepsilon}{|u|}\,\mathrm{d}^2u
 =\frac{\pi C}{2}.
\]
Consequently,
\[
 \cE(\mu_\varepsilon)\leq I_+(\mu_\varepsilon)
 \leq L_\varepsilon
 \Prob(|\mathsf{Z}_\varepsilon-\mathsf{W}_\varepsilon|\leq\delta)
 +\log\frac1\delta+\frac{\pi C}{2}.
\]
Weak convergence of the product measures and the Portmanteau theorem,
applied to the closed set $\{(z,w):|z-w|\leq\delta\}$, give
\[
 \limsup_{\varepsilon\downarrow0}
 \Prob(|\mathsf{Z}_\varepsilon-\mathsf{W}_\varepsilon|\leq\delta)
 \leq(\alpha\otimes\alpha)\{(z,w):|z-w|\leq\delta\}
 =\Col(\alpha).
\]
Division by $L_\varepsilon$ proves the claim.
\end{proof}

The next lemma turns mass in small disjoint disks into a lower bound
for logarithmic energy. It will quantify the contribution of the two
clusters arising from the all-common configuration in dimension three.
\begin{lem}\label{LemClusterLower}
Let $\mu$ be a probability measure on $\C$ with finite second moment.
If $\mathscr{D}_1,\ldots,\mathscr{D}_k$ are pairwise disjoint disks
of common radius $0<r<1/2$, then
\begin{equation}\label{EqClusterLower}
 \cE(\mu)\geq
 \sum_{j=1}^k\mu(\mathscr{D}_j)^2\log\frac1{2r}-m_2(\mu).
\end{equation}
\end{lem}

\begin{proof}
On each $\mathscr{D}_j\times\mathscr{D}_j$, the positive logarithmic
kernel is at least $\log(1/(2r))$. The product sets are disjoint, so
summing their contributions bounds $I_+(\mu)$ from below.
Subtract \eqref{EqNegativePart}. The same argument holds in the
extended sense if $I_+(\mu)=+\infty$.
\end{proof}

\section{The semicircular Wigner counterexample}\label{SectionWigner}

We first record a lower bound on the energy excess above the semicircle
law using the effective potential outside its support.
It will turn the positive-probability eigenvalue event in the proof of
Theorem~\ref{ThmWignerMain} into an explicit energy gap. Put
\[
 U_\sigma(x)=\int_\R\log\frac1{|x-y|}\,\sigma(\mathrm{d}y),
 \qquad
 v_{\mathrm{sc}}(x)=
 \begin{cases}
  0,&|x|\leq2,\\
  \displaystyle\int_{[2,|x|]}\sqrt{r^2-4}\,\mathrm{d}r,&|x|>2.
 \end{cases}
\]
The standard semicircle potential formula gives
$U_\sigma(x)=1/2-x^2/4$ for $|x|\leq2$; see
\cite[Exercise~2.6.4]{AndersonGuionnetZeitouni2010}.
For $x>2$, the Stieltjes transform
\cite[(2.4.6)]{AndersonGuionnetZeitouni2010} gives
$U_\sigma'(x)=-(x-\sqrt{x^2-4})/2$.
Integration from $2$, continuity, and evenness therefore yield
\[
 \frac{x^2}{2}+2U_\sigma(x)=1+v_{\mathrm{sc}}(x),
 \qquad x\in\R.
\]
If $\mu$ is a probability measure on $\R$ with compact support and a
bounded density, all the following integrals are absolutely convergent.
Positivity of logarithmic energy
for zero-mass signed measures
\cite[proof of Lemma~2.6.2(d), (2.6.20)]{AndersonGuionnetZeitouni2010}
and expansion around $\sigma$ give
\begin{align}
 \cF_{\R}(\mu)-\frac34
 &=\int_\R v_{\mathrm{sc}}\,\mathrm{d}\mu
   +\iint_{\R\times\R}\log\frac1{|x-y|}
        \,(\mu-\sigma)(\mathrm{d}x)(\mu-\sigma)(\mathrm{d}y)
 \notag\\
 &\geq\int_\R v_{\mathrm{sc}}\,\mathrm{d}\mu.
 \label{EqSemicircleGap}
\end{align}

\begin{proof}[Proof of Theorem~\ref{ThmWignerMain}]
The semicircle law is centered with second moment one, so the array
has the required Wigner law and
\[
 m_2(\rho_n)=\frac1{n^2}\E[\Tr(\mathbf{M}_n^2)]
 =\frac1{n^2}\left(n+2\binom n2\right)=1.
\]
Each entry lies in $[-2,2]$, hence
$\|\mathbf{M}_n/\sqrt n\|_{\mathrm{op}}
\leq\|\mathbf{M}_n/\sqrt n\|_{\mathrm F}\leq2\sqrt n$.

For completeness, boundedness of the mean spectral density follows
from a common diagonal shift. Write the diagonal entries as
$d_n=t$ and $d_i=t+u_i$ for $i<n$, a change of variables with Jacobian
one. For fixed $u$ and fixed off-diagonal entries, the $k$th ordered
eigenvalue of $\mathbf{M}_n/\sqrt n$ is $t/\sqrt n+\beta_k$.
The joint diagonal density is at most $\pi^{-n}$, and its support
forces $u\in[-4,4]^{n-1}$. Integrating over $t$, then over $u$ and the
off-diagonal entries, gives for every Borel set $B\subset\R$ of finite
Lebesgue measure
\[
 \Prob\bigl(\lambda_k(\mathbf{M}_n/\sqrt n)\in B\bigr)
 \leq\sqrt n\,8^{n-1}\pi^{-n}|B|.
\]
Thus $\rho_n$ has a bounded density, including when $n=1$.
The energy is finite by the bounds in Section~\ref{SectionPreliminaries}.

Clearly $\rho_1=\sigma$. Write
$\mathbf{M}_2=\left(\begin{smallmatrix}X&Y\\Y&Z\end{smallmatrix}\right)$
and let $A=\{X,Y,Z\in(3/2,2)\}$, which has probability $\pentry_0^3$.
The Rayleigh quotient at $(1,1)^{\mathsf T}/\sqrt2$ gives, on $A$,
\[
 \lambda_{\max}(\mathbf{M}_2/\sqrt2)
 \geq\frac{X+Z+2Y}{2\sqrt2}>\frac3{\sqrt2}>2.
\]
In particular, $\rho_2((2,\infty))\geq\pentry_0^3/2>0$.
Since $v_{\mathrm{sc}}$ is nonnegative and increasing on $(2,\infty)$,
\eqref{EqSemicircleGap} implies
\[
 \cF_{\R}(\rho_2)-\cF_{\R}(\rho_1)
 \geq\frac{\pentry_0^3}{2}v_{\mathrm{sc}}(3/\sqrt2)
 =\frac{\pentry_0^3}{2}\left(\frac34-\log2\right)>0.
\]
The last equality follows by integrating $\sqrt{r^2-4}$; positivity
also follows directly from the defining integral for $v_{\mathrm{sc}}$.
\end{proof}

\section{The Gaussian-regularised Bernoulli counterexample}
\label{SectionComplex}

\subsection{Regularity and the upper bound}

We first establish the regularity assertions of
Theorem~\ref{ThmComplexMain} and a general small-noise upper bound.
The latter will reduce the dimension-two estimate to a bound on atom
masses. Use $\pentry,\qentry,\mathsf{X}$ as in
Theorem~\ref{ThmComplexMain}.
Directly, $\E\mathsf{X}=0$ and $\E|\mathsf{X}|^2=1$. Therefore
$\E\mathsf{Y}_\varepsilon=0$ and
$\E|\mathsf{Y}_\varepsilon|^2=1$. Its density is
\[
 \frac{1+\varepsilon^2}{\pi\varepsilon^2}
 \left[
 \qentry\exp\left(-\frac{|\sqrt{1+\varepsilon^2}\,z+1/\sqrt{255}|^2}
                            {\varepsilon^2}\right)
 +\pentry\exp\left(-\frac{|\sqrt{1+\varepsilon^2}\,z-\sqrt{255}|^2}
                            {\varepsilon^2}\right)
 \right].
\]
This Gaussian mixture is strictly positive, bounded, and real analytic
on $\R^2$. Since $\mathsf{X}$ is bounded and $|\mathsf{G}|^2$ is
exponential with mean one,
\[
 |\mathsf{Y}_\varepsilon|^2
 \leq\frac{2(255+\varepsilon^2|\mathsf{G}|^2)}{1+\varepsilon^2}
\]
implies a finite exponential-square moment for sufficiently small
positive exponent. Thus the entry law has sub-Gaussian tails and
moments of every order.

Realise all matrices using an i.i.d.\ Bernoulli array
$\mathbf{X}_n$ and an independent standard complex Ginibre array
$\mathbf{G}_n$, so that
\[
 \frac{\mathbf{M}_{n,\varepsilon}}{\sqrt n}
 =\frac{\mathbf{X}_n+\varepsilon\mathbf{G}_n}
        {\sqrt{n(1+\varepsilon^2)}}.
\]
The orthogonal decomposition of the Gaussian matrix into its scalar
and trace-free parts is
\[
 \mathbf{G}_n=\frac{\mathsf{g}_n}{\sqrt n}\mathbf{I}+\mathbf{R}_n,
 \qquad
 \mathsf{g}_n=\frac1{\sqrt n}\Tr\mathbf{G}_n.
\]
Here $\mathsf{g}_n$ is standard circular Gaussian and is independent
of $(\mathbf{R}_n,\mathbf{X}_n)$. Scalar translation of eigenvalues
therefore makes $\nu_{n,\varepsilon}$ a convolution with the law of
$\varepsilon\mathsf{g}_n/(n\sqrt{1+\varepsilon^2})$. In particular,
it has density $f_{n,\varepsilon}$ satisfying
\begin{equation}\label{EqSpectralDensity}
 \|f_{n,\varepsilon}\|_\infty
 \leq\frac{n^2(1+\varepsilon^2)}{\pi\varepsilon^2}.
\end{equation}
Schur triangularisation gives
$\sum_j|\lambda_j(\mathbf{A})|^2\leq\|\mathbf{A}\|_{\mathrm F}^2$,
and hence
\begin{equation}\label{EqSpectralMoment}
 m_2(\nu_{n,\varepsilon})
 \leq\frac1{n^2}\E\|\mathbf{M}_{n,\varepsilon}\|_{\mathrm F}^2=1.
\end{equation}
Thus all these measures have finite logarithmic energy.

For fixed $n$, continuity of the roots of a monic polynomial as an
unordered multiset, followed by bounded convergence against bounded
continuous test functions, gives
\[
 \nu_{n,\varepsilon}\ \overset{\mathrm{w}}{\longrightarrow}\
 \nu_n^0\overset{\mathrm{def}}{=}\E[\mu_{\mathbf{X}_n/\sqrt n}]
 \qquad(\varepsilon\downarrow0).
\]
The limit has finite support. Applying Lemma~\ref{LemDensityUpper}
with \eqref{EqSpectralDensity}, and using \eqref{EqSpectralMoment}
to discard the bounded quadratic term after division by
$\log(1/\varepsilon)$, yields
\begin{equation}\label{EqSmallNoiseUpper}
 \limsup_{\varepsilon\downarrow0}
 \frac{\cF_{\C}(\nu_{n,\varepsilon})}{\log(1/\varepsilon)}
 \leq\Col(\nu_n^0).
\end{equation}

The following estimate controls the squared atom masses in dimension
two, together with
\eqref{EqSmallNoiseUpper}, it gives the upper half of the final
comparison. Put $c=-1/\sqrt{255}$. In dimension two, the all-common configuration
$\mathbf{X}_2=c\mathbf{J}_2$, where $\mathbf{J}_n$ denotes the
all-ones matrix, has probability $\qentry^4$ and normalized spectrum
$\{0,c\sqrt2\}$. Thus
\[
 \nu_2^0=\qentry^4\alpha_0+(1-\qentry^4)\beta,
 \qquad
 \alpha_0=\frac12(\delta_0+\delta_{c\sqrt2}),
\]
for a finitely supported probability measure $\beta$. Convexity of
$\Col$, together with $\Col(\alpha_0)=1/2$ and $\Col(\beta)\leq1$,
gives
\begin{equation}\label{EqDimensionTwoUpper}
 \Col(\nu_2^0)\leq1-\frac{\qentry^4}{2}
 \leq\frac12+2\pentry=\frac{65}{128}.
\end{equation}
The second inequality uses $(1-\pentry)^4\geq1-4\pentry$.

\subsection{The lower bound and comparison}

We now obtain the complementary lower bound from the two spectral
clusters of the all-common configuration in dimension three. This will
exceed the dimension-two upper bound on the logarithmic scale.
In dimension three, the all-common configuration has probability
$\qentry^9$. Its normalized matrix
$\mathbf{A}=c\mathbf{J}_3/\sqrt3$ is normal with spectrum
$\{0,0,\lambda_\star\}$, where $\lambda_\star=c\sqrt3\neq0$.
Fix $R\geq1$ and put
\[
 P_R=\Prob(\|\mathbf{G}_3\|_{\mathrm{op}}\leq R).
\]
On this Gaussian event and the all-common Bernoulli event, the
regularised normalized matrix $\mathbf{A}_\varepsilon$ satisfies,
for $0<\varepsilon<1$,
\[
 \|\mathbf{A}_\varepsilon-\mathbf{A}\|_{\mathrm{op}}
 \leq\frac{\varepsilon^2}{2}\|\mathbf{A}\|_{\mathrm{op}}
      +\frac{\varepsilon R}{\sqrt3}
 \leq K_R\varepsilon,
 \qquad
 K_R=\frac{\|\mathbf{A}\|_{\mathrm{op}}}{2}+\frac R{\sqrt3}.
\]
Set $r_\varepsilon=2K_R\varepsilon$.
For sufficiently small $\varepsilon$ (depending on $R$),
$r_\varepsilon<\min\{1/2,|\lambda_\star|/3\}$.
The normal-matrix resolvent bound implies that, for every $t\in[0,1]$,
the spectrum of
$\mathbf{A}+t(\mathbf{A}_\varepsilon-\mathbf{A})$ lies within
$K_R\varepsilon$ of $\{0,\lambda_\star\}$.
Continuity of the eigenvalue multiset along this path therefore keeps
two eigenvalues in $\mathscr{D}(0,r_\varepsilon)$ and one in
$\mathscr{D}(\lambda_\star,r_\varepsilon)$, counted with algebraic
multiplicity.

The Bernoulli and Gaussian events are independent. Taking expected
empirical measures gives
\[
 \nu_{3,\varepsilon}(\mathscr{D}(0,r_\varepsilon))
 \geq\frac23\qentry^9P_R,
 \qquad
 \nu_{3,\varepsilon}(\mathscr{D}(\lambda_\star,r_\varepsilon))
 \geq\frac13\qentry^9P_R.
\]
The disks are disjoint. Since $\cF_{\C}=\cE+m_2$,
Lemma~\ref{LemClusterLower} gives
\[
 \cF_{\C}(\nu_{3,\varepsilon})
 \geq\frac59\qentry^{18}P_R^2
       \log\frac1{4K_R\varepsilon}.
\]
Divide by $\log(1/\varepsilon)$ and let $\varepsilon\downarrow0$
with $R$ fixed. Then let $R\uparrow\infty$, so that $P_R\uparrow1$.
We obtain
\begin{equation}\label{EqDimensionThreeLower}
 \liminf_{\varepsilon\downarrow0}
 \frac{\cF_{\C}(\nu_{3,\varepsilon})}{\log(1/\varepsilon)}
 \geq\frac59\qentry^{18}
 \geq\frac59(1-18\pentry)=\frac{595}{1152}.
\end{equation}

\begin{proof}[Proof of Theorem~\ref{ThmComplexMain}]
The regularity, normalization, and finite-energy assertions were
proved above. Combining \eqref{EqSmallNoiseUpper},
\eqref{EqDimensionTwoUpper}, and \eqref{EqDimensionThreeLower}, we have
\[
 \limsup_{\varepsilon\downarrow0}
 \frac{\cF_{\C}(\nu_{2,\varepsilon})}{\log(1/\varepsilon)}
 \leq\frac{65}{128}=\frac{585}{1152}
 <\frac{595}{1152}
 \leq
 \liminf_{\varepsilon\downarrow0}
 \frac{\cF_{\C}(\nu_{3,\varepsilon})}{\log(1/\varepsilon)}.
\]
The positive separation of these bounds, and
$\log(1/\varepsilon)>0$ for $\varepsilon<1$, give the required strict
inequality for every sufficiently small positive $\varepsilon$.
\end{proof}

\bibliographystyle{acm}
\bibliography{non_monotonicity_entropy_v2}

\end{document}